\documentclass[12pt]{article}
\usepackage{e-jc}

\usepackage{amsmath,graphicx}
\usepackage{amssymb}
\usepackage{verbatim}
\usepackage{array}
\usepackage{latexsym}
\usepackage{enumerate}
\usepackage{amsfonts}
\usepackage{amsthm}
\usepackage{graphicx,stackengine,scalerel}
\usepackage[numbers,sort&compress]{natbib}
\usepackage[english]{babel}

{\theoremstyle{definition}
	\newtheorem*{definition*}{Definition}
	
	\newtheorem*{proposition*}{Proposition}
	\newtheorem*{corollary*}{Corollary}
	\newtheorem*{lemma*}{Lemma}

	\def\F{\mathbb F}

	\def\cX{\mathcal X}

	\def\PG{{\rm{PG}}}

	\def\AG{{\rm{AG}}}
	\def\GF{{\rm{GF}}}

	\def\dim{\mbox{\rm dim}}
	
	\def\fq{{\mathbb F}_q}
	\def\Fq{{\mathbb F}_q}
	
	\def\fqn{{\mathbb F}_{q^n}}

	\def\tang{\ThisStyle{\abovebaseline[0pt]{\scalebox{-1}{$\SavedStyle\perp$}}}}
	
\dateline{Aug 1, 2019}{May 8, 2020}{TBD}


\Copyright{The authors. Released under the CC BY-ND license (International 4.0).}

\title{On absolute points of correlations in   $\PG(2,q^n)$}

\author{Jozefien D'haeseleer \thanks{Supported by FWO grant.}\\
\small Department of Mathematics: Analysis, Logic and Discrete Mathematics \\[-0.8ex]
\small Ghent University\\[-0.8ex] 
\small Gent, Flanders, Belgium\\
\small\tt jozefien.dhaeseleer@ugent.be\\
\and
Nicola Durante \\
\small Dipartimento di Matematica e Applicazioni "Caccioppoli"\\[-0.8ex]
\small Universit\`a di Napoli
"Federico II", Complesso di Monte S. Angelo\\[-0.8ex]
\small Napoli, Italy\\
\small\tt ndurante@unina.it }

\begin{document}

\maketitle


\begin{abstract}
\noindent Let $V$ be a  $(d+1)$-dimensional vector space over a field $\F$.
Sesquilinear forms over $V$ have been largely studied when they are reflexive and hence give rise to a (possibly degenerate) polarity of  the $d$-dimensional projective space $\PG(V)$.  Everything is known in this case for both degenerate and non-degenerate reflexive forms if  $\F$  is either  ${\mathbb R}$, 
${\mathbb C}$ or a finite field  ${\mathbb F}_q$.  
In this paper we consider  degenerate, non-reflexive sesquilinear forms of $V=\fqn^3$.
We will see that these forms give rise to degenerate correlations of $\PG(2,q^n)$ whose
set of absolute points are, besides cones,  the (possibly degenerate)
$C_F^m$-sets studied in \cite{DonDur14}.
In the final section we collect some  results from the huge work of B.C. Kestenband  regarding what is known for the set of 
the absolute  points  of correlations in $\PG(2,q^n)$ induced  by a  non-degenerate, non-reflexive sesquilinear form of $V=\fqn^3$.
\end{abstract}
\section{Introduction and definitions}
\indent
Let $V$ and $W$ be two vector spaces over the same field ${\mathbb F}$.
A map $f:V\longrightarrow W$ is called $\sigma$-{\em semilinear}  if there exists
an automorphism $\sigma$ of ${\mathbb F}$
such that 
\[f(v+v') = f(v)+f(v') \, \, \, \, \mbox{ and } \, \, \, \, f(a v) = a^\sigma f(v)\]
for all vectors $v\in V$ and all scalars $a\in {\mathbb F}$.
\\ If $\sigma $ is the identity map, then $f$ is a {\em linear map}. 
\\
Let $V$ be an ${\mathbb F}$-vector space with  dimension $d+1$ and let $\sigma$ be an automorphism of   ${\mathbb F}$.
A map
\[ \langle \mbox{ } ,\mbox{ }   \rangle : (v,v') \in V\times V \longrightarrow \langle v , v' \rangle \in {\mathbb F}
\] is a $\sigma$-{\em sesquilinear form} or a $\sigma$-{\em semibilinear form} on $V$ if it is a linear map on the first argument and it is a $\sigma$-semilinear map on the second argument, that is:
\[\langle v+v', v'' \rangle = \langle v ,  v'' \rangle + \langle   v', v''\rangle\]
\[\langle v, v'+v'' \rangle = \langle v ,  v' \rangle + \langle   v, v''\rangle\]
\[\langle a v, v' \rangle= a \langle v, v' \rangle,  \,  \,  \,  \,  \langle  v, a v' \rangle= a^\sigma  \langle v, v' \rangle \]
for all $v,v',v''\in V$, $a\in {\mathbb F}$.
If $\sigma$ is the identity map, then $\langle \mbox{ },\mbox{ } \rangle$ is a {\em bilinear form}.
If ${\cal B} = (e_1,e_2,\ldots, e_{d+1})$ is an ordered basis of $V$, then for $x,y \in V$ we have
$\langle x,y \rangle = X_t A Y^\sigma$. Here $A=(a_{ij})=(\langle e_i, e_j \rangle)$ is
the {\em associated matrix}  to the $\sigma$-sesquilinear form in the ordered basis ${\cal B}$, and $X$, $Y$ are the columns of coordinates of $x,y$ w.r.t. ${\cal B}$.
In what follows  we will denote by $A_t$ the transpose
of a matrix $A$.
The term {\em sesqui} comes from the Latin and it means one and a half.
For every subset $S$ of $V$ put
\[ S^\perp := \{x \in V: \langle x, y \rangle=0 \,\,\, \forall y \in S \}\]
\[ S^{\tang} := \{y \in V: \langle x, y \rangle=0 \,\,\, \forall x \in S \}\]
Both $S^\perp$ and $S^{\tang}$ are subspaces of $V$. The subspaces $V^\perp$ and $V^{\tang}$ are called the {\em left} and the {\em right  radical} of $\langle , \rangle$, respectively.

\begin{proposition}
	The right and left radical of a sesquilinear form of a vector space $V$ have the same dimension.
\end{proposition}
\noindent Let $d+1$ be the dimension of $V$. Then  $\dim V^\perp = d+1 - \mbox{rk}(A)$ and  $\dim V^{\tang} = d+1 -\mbox {rk}(B)$
where $B=(a_{ij}^{\sigma^{-1}})$. The assertion follows from rk$(B)=\mbox{rk}(A)$.
\\

\noindent A sesquilinear form $\langle , \rangle$ is called  {\em non-degenerate} if  ${V}^\perp={V}^{\tang}=\{\underline{0}\}.$   It is called {\em reflexive} if $\langle x, y \rangle=0$ implies $\langle y, x \rangle =0$ for all $x,y \in V$.
\\
A  {\em duality} or  {\em correlation} of $\PG(d,\F)$ is a bijective map preserving incidence between points and hyperplanes of 
$\PG(d,\F)$.  It can be seen as a collineation of $\PG(d,\F)$ into its dual space $\PG(d,\F)^*$.
We will call {\em degenerate duality} or {\em degenerate correlation} of $\PG(d,\F)$ a non-bijective map 
preserving incidence between points and hyperplanes of 
$\PG(d,\F)$.
If  $V$ is equipped with a sesquilinear form $\langle \mbox{ },\mbox{ } \rangle$, then
there are two {\em induced}  (possibly degenerate) dualities  of $\PG(d,\F)$:
\[Y\mapsto X_t A Y^\sigma=0 \hspace{2cm} \mbox{ and } \hspace{2cm} Y\mapsto Y_t A X^\sigma=0.\] 
\\
The following holds.
\begin{theorem}
	Any duality of $\PG(d,\F), d\ge 2,$ is induced by a non-degenerate sesquilinear form on $V=\F^{d+1}$.
\end{theorem}

\noindent A duality applied twice gives a collineation of $\PG(d,\F)$. 
If $\langle \mbox{ },\mbox{ } \rangle$ is given by $\langle x, y \rangle = X_t A Y^\sigma$, then the associated collineation is the following map:
\[ f_{A_t^{-1}A^\sigma, \sigma^2}:  X \mapsto A_t^{-1} A^\sigma X^{\sigma^2}  \mbox{ (see \cite{HugPip}).}\]
The most important and studied types of dualities are those whose square is the identity, namely the {\em polarities}.
Note that in this case $A_t^{-1}A^\sigma=\rho I$ and $\sigma^2=1$.\\

\begin{proposition}\textnormal{\cite{Bal,Cam}}
	A duality is a polarity if and only if the sesquilinear form defining it is reflexive.
\end{proposition}

\noindent Non-degenerate, reflexive forms have been classified   long ago  (not only on finite fields but also on  ${\mathbb R}$ or  ${\mathbb C}$)  .
A reflexive $\sigma$-sesquilinear form $\langle , \rangle$ is:
\begin{itemize}
	\item  {\em $\sigma$-Hermitian} if $\langle y, x \rangle = \langle x , y \rangle^\sigma$, for all
	$x,y \in V.$ If $\sigma=1$ such a form is called {\em symmetric}. 
	\item
	{\em Alternating} if $\langle y , x \rangle = - \langle x , y \rangle$.
\end{itemize}

\begin{theorem}
	A non-degenerate, reflexive $\sigma$-sequilinear form is either alternating, or a scalar multiple of a $\sigma$-Hermitian 
	form. In the latter case, if $\sigma=1$ $($i.e. the form is  symmetric$)$, then the  scalar can be taken to be $1$.  
\end{theorem}

\noindent For a proof of the previous theorem we refer to \cite{Bal} or \cite{Cam}.
Note that in case of reflexive $\sigma$-sesquilinear forms, left- and right-radicals coincide. 
The set $\Gamma:X_t A X^\sigma=0$ of absolute points of a polarity is one of the well known objects in the projective space $\PG(d,\F)$: subspaces, quadrics and Hermitian varieties. 

\noindent We recall that  the classification theorem
for reflexive $\sigma$-sesquilinear forms also holds if the form is degenerate. 

\begin{remark}Let $\Gamma$ be the set of absolute points of a degenerate polarity in the projective space
	$\PG(d,\F)$, then $\Gamma$ is one of the following: subspaces, degenerate quadrics and degenerate Hermitian varieties. All of these sets are cones with vertex a subspace ${\cal V}$ of dimension
	$h$ (corresponding to $V^\perp$) and base a non-degenerate object of the same type in a subspace of dimension $d-1-h$
	skew with ${\cal V}$.
\end{remark}

\noindent  So everything is known for reflexive sesquilinear forms.
\\
Up to our knowledge (at least on finite fields) almost nothing is know on non-reflexive sesquilinear forms of $V=\fqn^{d+1}$.
Given a sesquilinear form of $V$ we may consider the induced correlation in $\PG(d,q^n)$ and the set $\Gamma$ of its {\em absolute}
points, that is the points $X$ such that $X\in {X}^\perp $ (or equivalently $X\in {X}^{\tang}$). 
If $\langle x, y \rangle = X_t A Y^\sigma$, then $\Gamma: X_t A X^\sigma =0$.
We will determine in the next sections the set $\Gamma$ in $\PG(1,q^n)$ and in $\PG(2,q^n)$.
\noindent 

Let $V=\fqn^{d+1}$, $d\in \{1,2\}$ and let $\Gamma:X_t A X^\sigma=0$, $\sigma \ne 1$, be the set of absolute points of a 
(possibly degenerate) correlation of $V$. 
We start with the following definition.
\begin{definition*}
	Let $\alpha$ be a $\sigma$-sesquilinear form of $\fqn^{d+1}$. A $\sigma$-{\em quadric} ($\sigma$-{\em conic} if $d=2$)
	of $\PG(d,q^n)$, $d\in \{1,2\}$ is the set of absolute points of the induced correlation of $\alpha$ in $\PG(d,q^n)$.
\end{definition*}

\noindent See  \cite{Dur} for the definition and the properties of $\sigma$-quadrics of $\PG(d,q^n)$ for every integer $d\ge 1$. It is an easy exercise (that we include in this paper) to determine the set $\Gamma$ in $\PG(1,q^n)$. 
For $d=2$ the set $\Gamma$  in $\PG(2,q^n)$ has been studied by B.C. Kestenband  in case $\langle \mbox{ } , \mbox{ }  \rangle$ is 
non-degenerate and the form is non-reflexive.  The results are contained in 10 different papers from 2000 to 2014.
We will  summarize some of his results in the last section.\\
Contrary to the reflexive case, we will see that in general the knowledge of the set $\Gamma$ of  absolute points of a correlation induced  by  a non-degenerate, non-reflexive sesquilinear  form will not help to determine the set $\Gamma$ in the degenerate case.
\\

\section{Absolute points of correlations of $PG(1,q^n)$}
In the sequel, we will determine the set of absolute points in $\PG(2,q^n)$  for a degenerate correlation induced by a degenerate $\sigma$-sesquilinear form
with associated automorphism $\sigma:x\mapsto x^{q^m}$, $(m,n)=1$, of $V=\fqn^3$.
\\
First assume that  $V= \fqn^2$  and consider the set of absolute points of a correlation induced by a  $\sigma$-sesquilinear form, that is the set of points in $\PG(1,q^n)$ given by $\Gamma:X_t AX^\sigma=0$, where 
\\
$$A=\left(
\begin{array}{cc}
a & b \\ c & d \end{array} \right) $$
\\
is a $2\times 2$-matrix over $\fqn$. Here again, let $\sigma$ be the automorphism $\sigma:x\mapsto x^{q^m}$, $(m,n)=1$, of $V=\fqn^2$.\\
The points of $\Gamma$ satisfy
\[ (ax_1+cx_2)x_1^\sigma+ (bx_1+dx_2)x_2^\sigma=0.\]

\begin{itemize}
	\item If $a=1, b=c=0, d\ne 0$, then $\Gamma:x_1^{\sigma+1}+dx_2^{\sigma+1}=0$.
	If $q$ is odd and $-d$ is a non-square of $\fqn$, then  $\Gamma=\emptyset$.
	If $q$ is even, $r=(q^n-1,q^m+1)$ and $d$ is an element of $\fqn$ such that $d^{\frac{q^n-1}{r}}\ne 1$, again 
	$\Gamma=\emptyset$. 
	\item If $a=1$, $b=c=d=0$, then $\Gamma:x_1^{\sigma+1}=0$ is the point $(0,1)$.
	\item If $b=1, a=c=d=0$, then $\Gamma:x_1x_2^\sigma=0$ is the union of the points $(1,0), (0,1)$.
	\item From the examples above, we see that $\Gamma$ can be the empty set, $\Gamma$ can be a point or the union of two points.  Hence we may assume that  $|\Gamma|>2$ and so we may suppose  that $(1,0),(0,1),(1,1)\in\Gamma$, therefore $a=d=0$, $b=-c$. So
	\[ \Gamma: x_1x_2(x_1^{\sigma-1}-x_2^{\sigma-1})=0.\]
	Hence $\Gamma$ is given by the union of $\{(1,0),(0,1)\}$ and the set of points $(x_1,x_2)$ in $\PG(1,q^n)$ such that $(\frac{x_2}{x_1})^{\sigma-1}=1$. Hence, $\Gamma$ is the set of points of
	an $\fq$-subline $\PG(1,q)$ of  $\PG(1,q^n)$.
\end{itemize}

\begin{lemma}
	The set of the absolute points in $\PG(1,q^n)$ of a $($possibly degenerate$)$ correlation induced by a $\sigma$-sesquilinear form of $\fqn^2$ is  one of the following:
	\begin{itemize}
		\item  the empty set, a point, two points;
		\item  an $\fq$-subline $\PG(1,q)$ of $\PG(1,q^n)$.
	\end{itemize}
\end{lemma}

\begin{definition}
	A $\sigma$-quadric $\Gamma$ of $\PG(1,q^n)$ is {\em non-degenerate} if $|\Gamma|\in \{0,2,q+1\}$,
	otherwise it is {\em degenerate}.
	A $\sigma$-conic $\Gamma$ of $\PG(2,q^n)$ is {\em non-degenerate} if  $V^\perp\cap V^{\tang} = \{\underline{0}\}$ 
	and $\Gamma$ does not contains lines. Otherwise, it is {\em degenerate}.
\end{definition}

\noindent Note that if $\Gamma$ is a degenerate $\sigma$-quadric of $\PG(1,q^n)$, then there exists a degenerate sesquilinear form such that $V^\perp = V^{\tang}$ with $\dim(V^\perp)=1$.

\noindent Next we determine the possible intersection configurations of a line and a $\sigma$-conic of $\PG(2,q^n)$. So let  $\Gamma$ be the set of absolute points in $\PG(2,q^n)$ of a (possibly degenerate) correlation induced by a $\sigma$-sesquilinear form of $\fqn^{3}$,  $\sigma:x\mapsto x^{q^m}, (m,n)=1$. Hence $\Gamma: X_t A X^\sigma=0$.

\begin{proposition}\label{pr_line}
	Every line intersects a $\sigma$-conic  $\Gamma$ of $\PG(2,q^n)$ in either 
	$0$ or $1$ or $2$ or $q+1$ points $($an $\fq$-subline$)$ or it is contained in $\Gamma$. 
\end{proposition}
\begin{proof}
Let $Y,Z$ be two distinct points of $\PG(2,q^n)$ and let $\ell: X=\lambda Y + \mu Z$, $(\lambda,\mu)\in \fqn^2\setminus \{(0,0)\}$ be the line containing them. The number of points of $\Gamma \cap \ell$ is given by the number of points of a
$\sigma$-quadric on the line $\ell$. Indeed a point $ X=\lambda Y + \mu Z\in \ell\cap \Gamma$ if and only if $a \lambda^{\sigma+1} + b \lambda \mu^\sigma + c \lambda^\sigma \mu
+ d \mu^{\sigma+1}=0$, where $a=Y_t A Y^\sigma, b=Y_t AZ^\sigma, c=Z_t A Y^\sigma, d=Z_t AZ^\sigma$. The assertion follows.
\end{proof}

\section{Steiner's projective generation of conics and $C_F^m$-sets}
Before determining the sets of absolute points of correlations of $\PG(2,q^n)$ we will recall some subsets of $\PG(2,q^n)$,
called (possibly degenerate) $C_F^m$-sets,
that have been introduced and studied in \cite{DonDur03,DonDur05,DonDur14}  generalizing the constructions of
conics due to J. Steiner. We will see that $\sigma$-conics of $\PG(2,q^n)$ with $|A|=0$, different form a cone with vertex a point and base a $\sigma$-quadric of $\PG(1,q^n)$ will be a (possibly degenerate) $C_F^m$-set.
First we recall J. Steiner's construction of conics. Let $\mathbb{F}$ be a field and let $\PG(2,\mathbb{F})$ be a Desarguesian projective plane.
Let $R$ and $L$ be two distinct points of $\PG(2,\mathbb{F})$ and let ${\cal P}_R$ and $ {\cal P}_L$ be the pencils of lines
with centers $R$ and $L$, respectively.  In 1832 in \cite{Ste} J. Steiner  proves the following:

\begin{theorem}
	The set of $\Gamma$ of points of intersection of corresponding lines under  a projectivity $\Phi : {\cal P}_R \longrightarrow {\cal P}_L$
	is one the following:
	\begin{itemize}
		\item if $\Phi(RL) \ne RL$, then $\Gamma$ is a non-degenerate conic;
		\item if $\Phi(RL) = RL$, then $\Gamma$ is a degenerate conic.
	\end{itemize}
\end{theorem}

\noindent Inspired by Steiner's projective generation of conics, in \cite{DonDur03,DonDur05,DonDur14}  the following definitions are given.

\begin{definition}
	The set $\Gamma \subset \PG(2,q^n)$ of points of intersection of corresponding lines under  a collineation $\Phi : {\cal P}_R \longrightarrow {\cal P}_L$  with accompanying automorphism $x\mapsto x^{q^m}$ is one the following:
	\begin{itemize}
		\item if $\Phi(RL) \ne RL$, then $\Gamma$ is called a $C_F^m$-set with {\em vertices} $R$ and $L$;
		\item if $\Phi(RL) = RL$, then $\Gamma$ is  called a {\em degenerate} $C_F^m$-set with {\em vertices} $R$ and $L$. 
	\end{itemize}
\end{definition}

\noindent Moreover in \cite{DonDur14} the following is proved:

\begin{theorem}
	A $C_F $-set $\Gamma$ of $\PG(2,q^n)$ with vertices $R=(1,0,0)$ and $L=(0,0,1)$ has canonical equation $x_1x_3^{q^m} -x_2^{q^m+1}=0$. Hence $|\Gamma|=q^n+1$ and it is a set of type $(0,1,2,q+1)$ w.r.t. lines. It is the union of $\{R,L\}$ with $q-1$ pairwise disjoint scattered $\fq$-linear sets of pseudoregulus type, each of which is isomorphic to
	the set of points and lines of  a $\PG(n-1,q)$, that is called a {\em component} of $\Gamma$. 
\end{theorem}

\begin{theorem}
	A degenerate $C_F^m$-set of $\PG(2,q^n)$  with vertices $R=(1,0,0)$ and $L=(0,1,0)$ has canonical equation $x_3 (x_1x_3^{q^m-1} -x_2^{q^m})=0$. Here $|\Gamma|=2q^n+1$ and it is a set of type $(1,2,q+1,q^n+1)$ w.r.t. lines. It is the union of the line $RL$ with a set ${\cal A}$ of $q^n$ affine points  isomorphic to $\AG(n,q)$ whose directions on the line $RL$
	form a maximum scattered $\fq$-linear set ${\cal S}$ of pseudoregulus type with transversal points $R$ and $L$.
	Moreover   ${\cal A} \cup {\cal S}$ is a maximum scattered $\fq$-linear set of rank $n+1$, that is a
	blocking set of R\'edei type and vice versa.
\end{theorem}

\noindent We refer to \cite{DonDur14} and to \cite{Lun-Mar-Pol-Tro} for the definition and properties of the relevant objects in the previous two theorems such as maximum scattered $\fq$-linear set
and blocking set of R\'edei type.

\noindent In the next section we will see that the set of absolute points of a degenerate correlation of $\PG(2,q^n)$ induced by
a degenerate, non-reflexive $\sigma$-sequilinear form of $\fqn^3$ is either a $C_F^m$-set or a degenerate $C_F^m$-set
or a cone with vertex a point $R$ and base a $\sigma$-quadric of a line $\ell$ not through $R$.

\section{Absolute points of degenerate correlations of $\PG(2,q^n)$}

\noindent In this section we determine the possible structure for the set of  absolute points in $\PG(2,q^n)$ of a degenerate correlation induced by a $\sigma$- sesquilinear form in $\fqn^3$. Let $\Gamma:X_t A X^\sigma=0$ be such  a set of points.
For the sequel we can assume $\sigma \ne 1$. Indeed if $\sigma =1$, then $\Gamma$ is always a (possibly degenerate)
conic or the full pointset (a degenerate symplectic  geometry). In what follows, $\sigma$, if not differently specified, will denote always  the map $x\mapsto x^{q^m}, (m,n)=1$.
\\
First assume $rank(A)=2$, then $V^\perp$ and $V^{\tang}$ are one-dimensional vector spaces of $V$, so points 
of $\PG(2,q^n)$. If $V^\perp \ne V^{\tang}$, then  we may assume that the point $R=(1,0,0)$ is the right-radical and the point $L=(0,0,1)$ is the  left-radical. It follows that
\\
$$A=\left(
\begin{array}{ccc}
0 & a & b \\ 0 & c & d \\ 0 & 0 & 0 \end{array} \right) $$
\\
and 
\[\Gamma:X_tAX^\sigma = (ax_1+cx_2)x_2^\sigma+(bx_1+dx_2)x_3^\sigma=0.\]
Consider the pencil of lines through the point $R$, that is 
\[{\cal P}_R=\{\ell_{\alpha,\beta}:(\alpha,\beta) \in \fqn\times\fqn\setminus \{(0,0\}\},   \mbox{ where }\] 
\[\ell_{\alpha,\beta}: 
\left\{
\begin{array}{ccc}
x_1 & = & \lambda \\
x_2 & = & \mu \alpha \\
x_3 & = & \mu \beta \end{array}, (\lambda,\mu) \in \fqn\times\fqn\setminus \{(0,0)\}\right.\]
and the pencil of lines through $L$, that is 
\[{\cal P}_L=\{\ell'_{\alpha',\beta'}:   (\alpha',\beta') \in \fqn\times\fqn\setminus \{(0,0\}\},  \;\;\; \mbox{ where } \;\;\;    \ell'_{\alpha',\beta'}: \alpha' x_1+\beta'x_2=0.\ \] 
Let 
\[\Phi: {\cal P}_R \longrightarrow {\cal P}_L \] 
be the collineation between ${\cal P}_R$ and ${\cal P}_L$ given by $\Phi(\ell_{\alpha,\beta}) = \ell'_{\alpha',\beta'}$, where
\[ (\alpha',\beta')_t = A' (\alpha, \beta)_t^\sigma.	\]
where $A'$ is the matrix
\\
$$A'=\left(
\begin{array}{cc}
a & b \\ c & d \end{array} \right). $$
\\
Note that $|A'|\ne 0$ since $rank(A)=2$. It is easy to see that $\Gamma$ is the set of points of intersection of corresponding lines under the collineation $\Phi$ and hence it is a  (possibly degenerate) $C_F^m$-set (see \cite{DonDur03,DonDur05,DonDur14}).
\\
Let $Y=(y_1,y_2,y_3)$ be a point of $\Gamma$, then the tangent line $t_Y$ to $\Gamma$ at the point $Y$ is the
line 
\[t_Y: (ay_2^\sigma+by_3^\sigma)x_1+(cy_2^\sigma+dy_3^\sigma)x_2=0.\]
So for every point $Y$ of $\Gamma$ the line $t_Y$ contains the point $L=(0,0,1)$.
The tangent line $t_L$ to $\Gamma$ at the point $L$ is the line
\[t_L:bx_1+dx_2=0.\]
First assume that $t_L=RL:x_2=0$, then $b=0$ and we can put $d=1$ obtaining
\[\Gamma: (ax_1+cx_2)x_2^\sigma+x_2x_3^\sigma=0,\] with $a\ne 0$, so
\[\Gamma: x_2(ax_1x_2^{\sigma-1}+cx_2^\sigma+x_3^\sigma)=0,\]
that is  a degenerate $C_F^m$-set
with vertices $R(1,0,0)$ and $L(0,0,1)$ since the collineation $\Phi$ maps the line $RL$ into itself (see \cite{DonDur14}).
\\
Next assume $t_L$ is not the line $RL$, so we may suppose, w.l.o.g., that $t_L:x_1=0$, that is $d=0$ and $b=1$.
In this case 
\[ \Gamma:(ax_1+cx_2)x_2^\sigma +x_1x_3^\sigma=0,\] with $c\ne 0$,
is a non-degenerate $C_F^m$-set with vertices $R(1,0,0)$ and $L(0,0,1)$ since the collineation $\Phi$ does not map the line $RL$ into itself (see \cite{DonDur14}).
\\
Next assume that $V^\perp = V^{\tang}$ so they coincide as projective points of $\PG(2,q^n)$. We may assume that
$R=L=(1,0,0)$ is both the left and right radical of $\langle \mbox{ },\mbox{ } \rangle$.
In this case the set of absolute points is the set $\Gamma$ of points 
of $\PG(2,q^n)$ such that $X_t A X^\sigma =0$ with
\\
$$A=\left(
\begin{array}{ccc}
0 & 0 & 0 \\ 0 & a & b \\ 0 & c & d \end{array} \right) $$
\\
so 
\[\Gamma: (ax_2+bx_3)x_2^\sigma+(cx_2+dx_3)x_3^\sigma=0.\]
It follows that $\Gamma$
is the set of points of a cone with vertex the point $R$,  the base of this cone is either the empty set, a point, two points or $q+1$ points of an $\fq$-subline of a line not through the point $R$.
\\
Note that $\Gamma$ can be seen as the set of points of intersection of corresponding lines under the same collineation $\Phi$  from ${\cal P}_R$ to ${\cal P}_L$, with $R=L$, similar to the case $R\ne L$.\\
Next assume $rank(A)=1$. In this case dim $V^\perp =$ dim $V^{\tang} =2$, so in $\PG(2,q^n)$ the rigth- and left-radical are given by two lines $r$ and $\ell$.
First assume $r\ne \ell$, so we may put $r:x_3=0$ and $\ell:x_1=0$, then $\Gamma:x_1x_3^\sigma=0$, that is the union of the two lines $r$ and $\ell$.
\\
Finally assume that  $r=\ell$, e.g. $r=\ell:x_3=0$, then 
$\Gamma:(cx_1+dx_2+ex_3){x_3}^\sigma=0$ 
is again the union of two (possibly coincident) lines.

\noindent In the next proposition we summarize what has been proved  with the previous arguments.

\begin{proposition}
	The set of absolute points in $\PG(2,q^n)$ of a degenerate correlation induced by a degenerate $\sigma$-sesquilinear form of $\fqn^3$ is one of the following:
	\begin{itemize}
		\item a cone with vertex a point $R$ and base  a $\sigma$-quadric of a line $\ell$ not through $R$
		(i.e. just $R$, a line, two lines or an $\fq$-subpencil of ${\cal P}_R$);
		\item a degenerate $C_F^m$-set $($see \cite{DonDur14} $)$;
		\item a  $C_F^m$-set $($see \cite{DonDur14}$).$ 
	\end{itemize}
\end{proposition}

\subsection{Some applications for $C_F^m$-sets of $\PG(2,q^n)$}
In recent years,  (degenerate or not) $C_F^m$-sets of $\PG(2,q^n)$ have been used for several applications.
We recall some of them here.  \\
Let $M_{m,n}(\fq)$, with $m\le n$, be the vector space of all the $m\times n$
matrices with entries in $\fq$. The {\em distance} between
two matrices is the rank of their difference. An $(m \times n,q,s)$-{\em rank
	distance code} is a subset ${\cal X}$ of $M_{m,n}(\fq)$ such that the minimum
distance between any two of its distinct elements is $s$. An  {\em $\fq$-linear}
$(m\times n,q,s)$-rank distance code is a subspace of
$M_{m,n}(\fq)$. 

\noindent It is known (see e.g. \cite{Del}) that the size of an $(m\times n,q,s)$-rank distance code
$\cX$ satisfies the {\em Singleton-like bound}: 
\[
|\cX| \le q^{n(m-s+1)}.
\]
When this bound is achieved, $\cX$ is called an $(m\times n, q, s)${\em -maximum
	rank distance code}, or $(m\times n, q, s)$-{\em MRD code}, for short.

\noindent In finite geometry $(m\times m,q,m)$-MRD codes are known as {\em spread sets} (see
e.g. \cite{Dem}) and there are examples for both $\fq$-linear and non-linear codes. \\
The first class of non-linear MRD codes, different from spread sets, are the 
non-linear $(3\times 3,q,2)$-MRD codes constructed 
in \cite{CosMarPav}  by using $C_F^1$-sets of $\PG(2,q^3)$.

\noindent
Indeed starting with a $C_F^1$-set ${\cal C}$ of $\PG(2,q^3)$ in \cite{CosMarPav}, they construct a set ${\cal E}$ of 
$q^3+1$ points that is an {\em exterior} set  to the component ${\cal C}_1$ of ${\cal C}$ (i.e. the lines joining any two distinct points of ${\cal E}$ is external to ${\cal C}_1$).
\\
By using field reduction from $\PG(2,q^3)$ to $\PG(8,q) = \PG(M_{3,3}(q))$, the component 
${\cal C}_1\cong \PG(2,q)$ corresponds to the Segre variety ${\cal S}_{2,2}$. This is the set of matrices with rank $1$.   The set ${\cal E}$ corresponds to
an exterior set to  the Segre variety ${\cal S}_{2,2}$ of $\PG(8,q)$ and hence a set of $q^6$ matrices such that  the difference 
between any two has rank at least two, i.e. a $(3\times 3,q,2)$-MRD code.
\\
In this section, it is shown that, starting from a $C_F^m$-set of $\PG(2,q^n)$,   infinite families of
non-linear $(3\times n,q,2)$-MRD codes can be constructed. (See also \cite{Dur}).
\\
Let $R=(1,0,0)$ and $ L=(0,0,1)$ be two points of $\PG(2,q^n)$
with $n \ge 3$ and let ${\cal C}$ be the $C_F^m$-set with vertices $R$ and $L$ given by 
\[ 
{\cal C}=
\{P_t=(t^{q^m+1},t,1):
t\in \fqn\} \cup \{R\}.
\]
It follows that 

$$ {\cal C} = {\displaystyle{\bigcup}_{a\in\Fq^*}\; {\cal C}_a \cup \{R,L\}}, $$ with 
${\cal C}_a=\{P_t: t\in N_a\}$ and   $N_a=\{x \in \fqn: N(x)=a\}$,
where $N(x)=x^{\frac{q^n-1}{q-1}}$ denotes the {\em norm} of an element $x\in \fqn$ w.r.t. $\fq$.
For every $a\in\fq^*$, consider the partition of the points  of the line $RL$, different from $R$ and $L$, into subsets
$J_a=\{( -t,0,1):t\in N_a\}$ and let $\pi'_1\cong \PG(2,q)$ be a
subgeometry contained in  ${\cal C}_1$.

\begin{theorem}
	For every subset $T$ of $\Fq^*$ containing $1$, the set 
	$$
	\cX=({\cal C}\setminus {\textstyle\bigcup}_{a\in T} {\cal C}_a) \cup 
	{\textstyle\bigcup}_{a\in T} J_a
	$$ is an exterior set with respect to 
	$\pi_1'$.
\end{theorem}
\begin{proof} The lines $RP$ and $LP$, with 
$P\in {\cal C}\setminus ({\textstyle\bigcup}_{a\in T} {\cal C}_a \cup \{R,L\})$, meeting ${\cal C}$ exactly in two points,
are external lines w.r.t.  $\pi'_1$. 
\\
Similarly, for every $P\in {\cal C}_b$ and $P'\in {\cal C}_{b'}$ with 
$b,b'\in \Fq^*\setminus T$ the line $PP'$ is external to  $\pi'_1$. 
\\
Finally, for every  point $P\in {\cal C}_b$  with $b\in \Fq^*\setminus T$ and $P'\in J_{a}$ with $a\in  T$,
the line $PP'$ is external to  $\pi'_1$. 
\\
Indeed suppose, by way of contradiction, that 
the line $PP'$ meets $\pi'_1$ in a point $S$ with coordinates 
$(x^{q^{2m}},x^{q^m},x)$. Let
$P=(\alpha^{q^m+1},\alpha,1)$ with
$N(\alpha)=b$ and let $P'=(-t,0,1)$ with $N(t)=a$. By
calculating the determinant of the matrix $M$ whose rows are the coordinates
of the points $S,P,P'$, we have that $|M|=-t M_1 +
\alpha M_1^{q^m}$, where $M_1$ is the cofactor of the element $m_{3,1}$ of $M$.
Since $S, P$ are distinct points, $M_1\ne 0$,
hence $|M|=0$ if and only if $t=\alpha M_1^{q^m-1}$, that is a contradiction since
$N(\alpha)=b$ while $N(t)=a$ with $a\ne b$. \end{proof}  
\noindent From the previous theorem we have the following result.
\begin{corollary}
	For all $n \ge 3, q>2$, the vectors $\rho v\in
	M_{3,n}(q)$, $\rho \in \Fq^*$, whose corresponding points are in $\cX$, plus
	the zero vector, give a non-linear $(3\times n,q,2)$-MRD code.
\end{corollary}

\begin{remark}
	These codes have been generalized first,  by using bilinear forms and 
	the cyclic representation of $\PG(n-1,q^n)$ by N. Durante and A. Siciliano in  \cite{DurSic}, to non-linear
	$(n\times n, q, n-1)$-MRD codes. Later G. Donati and N. Durante generalized these codes in \cite{DonDur18}  by using $\sigma$-normal rational curves, which is a generalization of normal rational curves 
	in $\PG(d,q^n)$, to non-linear {$((d+1)\times n, q, d)$-MRD codes}, where $d\le n-1$.
\end{remark}

\noindent Other applications of (degenerate or not) $C_F^m$-sets of $\PG(2,q^n)$ are:
\begin{itemize}
	\item{\em  Conics of $\PG(2,q^n), q$ even.} \\ Indeed the affine part, plus the vertices, of a degenerate $C_F^1$-set of 
	$\PG(2,q^n), q=2$ form a regular hyperoval.
	\item  {\em Translation hyperovals of $\PG(2,q^n)$, $q$ is even.} \\ Indeed the affine part, plus the vertices of a $C_F^m$-set
	of $\PG(2,q^n), q=2$, $m>1$, form a translation hyperoval.
	\item {\em Semifield flocks of a quadratic cone of $\PG(3,q^n)$, $q$ odd
		(and hence also translation ovoids of the parabolic quadric $Q(4,q^n)$, $q$ odd). }\\
	Indeed, every semifield flock ${\cal F}$ of a quadratic cone of $\PG(3,q^n), q$ odd is associated with a scattered
	$\fq$-linear set ${\cal I}({\cal F})$  of internal points w.r.t. a non-degenerate conic  ${\cal C}$ of a plane $\PG(2,q^n)$ of $\PG(3,q^n)$.
	The known examples of semifield flocks of a quadratic cone of $\PG(3,q^n)$ are the linear flock, the Kantor-Knuth,
	the Cohen-Ganley $(q=3)$ and the sporadic semifield flock $(q=3,n=5)$. These semifield flocks give as set ${\cal I}({\cal F})$
	a point, a scattered $\fq$-linear set of pseudoregulus type on a secant line to ${\cal C}$, a component of a $C_F^1$ set of $\PG(2,3^n)$
	and a component of a $C_F^2$-set of $\PG(2,3^5)$, respectively.
\end{itemize}

\noindent See \cite{Dur} for more details on these applications.

\section{Non-degenerate $\sigma$-sesquilinear forms in $\fqn^3$}

\noindent In this section we will determine the possible structure for the set of  absolute points  in 
$\PG(2,q^n)$ of a correlation induced by a non-degenerate $\sigma$-sesquilinear from of $\fqn^3$. So consider the set of points $\Gamma:X_t A X^\sigma=0$.
These sets of points have been studied in several papers by B. Kestenband. 
\noindent  For this reason the set of points of this section have been called {\em Kestenband $\sigma$-conics} in 
\cite{Dur}. We will try to summarize what is know up to now.
First recall some easy properties for the set $\Gamma$.
For the sequel we can assume $\sigma \ne 1$, since if $\sigma =1$, then $\Gamma$ is always a (possibly degenerate)
conic.
\begin{proposition}
	Every line of $\PG(2,q^n)$ meets $\Gamma:X_t A X^\sigma=0$ in either $0$, $1$,  $2$ or $q+1$ points of an $\fq$-subline.
\end{proposition}
\begin{proof}  From Proposition \ref{pr_line}  we only have to show that $\Gamma$ does not contain lines.
Suppose, by way of contradiction, that  $\Gamma$ contains a line, say $\ell$, then we can assume w.l.o.g. that $\ell:x_3=0$.
Then the matrix $A$ assumes the following shape:
\\
$$A=\left(
\begin{array}{ccc}
0 & 0 & a \\ 0 & 0 & b \\ c & d & e \end{array} \right) $$
\\
hence $rank(A)\le 2$,  a contradiction.
\end{proof}
\noindent In the next two theorems let $\sigma:x\mapsto x^{q^m}$ with $(m,2n+1)=1$.

\begin{theorem}
	Let $\Gamma:X_t A X^\sigma=0$ be the set of absolute points of a correlation of $\PG(2,q^{2n+1})$,  $n>0$,
	$q$ odd. Then we find that  $|\Gamma|\in \{q^{2n+1}+\epsilon q^{n+1}+1 | \epsilon\in \{0,1,-1\}\}$ and the set $\Gamma$ depends on the number of points of $\Gamma$ and points outside $\Gamma$ fixed by the induced collineation $f_{A_t^{-1}A^\sigma,\sigma^2}$.
	The following holds:
	\begin{itemize}
		\item If $f_{A_t^{-1}A^\sigma,\sigma^2}$ fixes no points outside $\Gamma$, then $|\Gamma|=q^{2n+1}+1$. Moreover
		$f_{A_t^{-1}A^\sigma,\sigma^2}$ fixes one point of $\Gamma$ and $\Gamma$ is projectively equivalent to:
		\[x_3x_1^\sigma+x_2^{\sigma+1}+(x_1+ex_2)x_3^\sigma=0,\] where $e^{q^{2n}}+e^{q^{2n-1}}+\ldots+e^q+e\ne 0.$
		\item If $f_{A_t^{-1}A^\sigma,\sigma^2}$ fixes more than one point outside  $\Gamma$, then $|\Gamma|=q^{2n+1}+1$.  Moreover, $\Gamma$ is projectively equivalent to: 
		\[ x_1^{\sigma+1}+x_2^{\sigma+1}+x_3^{\sigma+1}=0,\]
		the collineation $f_{A_t^{-1}A^\sigma,\sigma^2}$ fixes $\PG(2,q)$ pointwise and hence it fixes the $q+1$ points of $\Gamma$ on the subconic 
		$x_1^2+x_2^2+x_3^2=0$ in $\PG(2,q)$ and $q^2$ points of $\PG(2,q)\setminus \Gamma$.
		\item If $f_{A_t^{-1}A^\sigma,\sigma^2}$ fixes one point outside  $\Gamma$, then $\Gamma$ is projectively equivalent to: 
		\[ax_1^{\sigma+1}+(bx_1+cx_2)x_2^\sigma+dx_3^{\sigma+1}=0,\] with $a,b,c,d\ne 0$. 
		Moreover $f_{A_t^{-1}A^\sigma,\sigma^2}$ fixes either
		$0$, $1$, $2$ or $q+1$ points of $\Gamma$. In particular
		\begin{itemize}
			\item If $f_{A_t^{-1}A^\sigma,\sigma^2}$ fixes either $0$, $2$ or $q+1$ points of $\Gamma$, then $\epsilon=0.$ In the last case the $q+1$ points of $\Gamma$ fixed by $f_{A_t^{-1},A^\sigma,\sigma^2}$ are collinear.
			\item If $f_{A_t^{-1}A^\sigma,\sigma^2}$ fixes $1$ point of $\Gamma$, then $\epsilon \in \{-1,1\}.$
		\end{itemize}
	\end{itemize}
\end{theorem}

\begin{theorem}
	Let $\Gamma:X_t A X^\sigma=0$ be the set of absolute points of a correlation of $\PG(2,q^{2n+1})$,  $n>0$,
	$q$ even. Then we find that  $|\Gamma|\in \{q^{2n+1}+\epsilon q^{n+1}+1 | \epsilon\in \{0,1,-1\}\}$ and the set $\Gamma$ depends on the number of points of $\Gamma$ and points outside $\Gamma$ fixed by the induced collineation $f_{A_t^{-1}A^\sigma,\sigma^2}$.
	The following holds:
	\begin{itemize}
		\item If $f_{A_t^{-1}A^\sigma,\sigma^2}$ fixes no points outside $\Gamma$, then $\epsilon\in \{-1,+1\}$. Moreover
		$f_{A_t^{-1}A^\sigma,\sigma^2}$ fixes one point of $\Gamma$ and $\Gamma$ is projectively equivalent to: 
		\[x_3x_1^\sigma+x_2^{\sigma+1}+(x_1+ex_2)x_3^\sigma=0,\] where $e^{q^{2n}}+e^{q^{2n-1}}+\ldots+e^q+e\ne 0$.
		\item If $f_{A_t^{-1}A^\sigma,\sigma^2}$ fixes at least  one point outside  $\Gamma$, then $\epsilon=0$.  Moreover, 
		$\Gamma$ is projectively equivalent to: 
		\[ x_1^{\sigma+1}+\rho x_1 x_2^{\sigma}+x_2^{\sigma+1}+x_3^{\sigma+1}=0,\]
		for some $\rho \in \mathbb{F}_{q^{2n+1}}$.
		The trinomial $x^{q^m+1}+\rho x + 1$ can have one, two, $q+1$ or no zeros. 
		\begin{itemize}
			\item[1.] If it has one or $q+1$ zeros, 
			then $\rho=0$ and  the collineation $f_{A_t^{-1}A^\sigma,\sigma^2}$ fixes $\PG(2,q)$ pointwise and hence it fixes the $q+1$ points of $\Gamma$ on the subline 
			$x_1+x_2+x_3=0$ in $\PG(2,q)$ and $q^2$ points of $\AG(2,q)=\PG(2,q)\setminus \Gamma$.
			\item[2.] If it has two zeros, then the collineation  $f_{A_t^{-1}A^\sigma,\sigma^2}$ fixes $2$  points of $\Gamma$.
			\item[3.] If it  has no zeros, then the collineation  $f_{A_t^{-1}A^\sigma,\sigma^2}$ fixes no point of $\Gamma$. 
		\end{itemize}
	\end{itemize}
\end{theorem}

\noindent In the next  theorems let $\sigma:x\mapsto x^{q^m}$ with $(m,2n)=1$.

\begin{theorem}
	Let $\Gamma:X_t A X^\sigma=0$, with $A$ a diagonal matrix, be the set of absolute points of a correlation of $\PG(2,q^{2n})$, $n>0$.  Then
	$$|\Gamma| \in \{ q^{2n}+1+(-q)^{n+1}(q-1) , q^{2n}+1 +(-q)^n(q-1),   q^{2n}+1-2(-q)^n\}.$$
	Moreover,
	\begin{itemize}
		\item if $|\Gamma|  = q^{2n}+1+(-q)^{n+1}(q-1)$, then $\Gamma$ is projectively equivalent to
		\[x_1^{\sigma+1}+x_2^{\sigma+1}+x_3^{\sigma+1}=0;\]
		\item if $|\Gamma|  = q^{2n}+1+(-q)^{n}(q-1)$, then $\Gamma$ is projectively equivalent to
		\[x_1^{\sigma+1}+x_2^{\sigma+1}+ a x_3^{\sigma+1}=0,\] with $a\notin \{x^{q+1}:x \in {\mathbb F}_{q^{2n}}\}; $
		\item if $|\Gamma|  = q^{2n}+1-2(-q)^{n}$, then  $\Gamma$ is projectively equivalent to
		\[ ax_1^{\sigma+1}+bx_2^{\sigma+1}+  x_3^{\sigma+1}=0,\]
		with $a,b,a/b\notin \{x^{q+1}:x \in {\mathbb F}_{q^{2n}}\} .$
	\end{itemize}
\end{theorem}

\begin{theorem}
	Let $\Gamma:X_t A X^\sigma=0$, with $A$ a  non-diagonal matrix,  be the set of absolute points of a correlation of $\PG(2,q^{2n})$, $n>0, q$ even.  Then $|\Gamma| \in \{ q^{2n}-(-q)^{n+1}+1, q^{2n}-(-q)^n+1,   q^{2n}+1. \}$
\end{theorem}

{\begin{theorem}
		Let $\Gamma:X_t A X^\sigma=0$, with $A$ a non-diagonal matrix, be the set of absolute points of a correlation of $\PG(2,q^{2n})$, $n>0$, $q$ odd.  Then we can distinguish two cases. 
		\begin{enumerate}
			\item $\Gamma$ is projectively equivalent to
			\[r x_1^{\sigma+1}+\rho x_1 x_2^\sigma+sx_2^{\sigma+1}+x_3^{\sigma+1}=0\]
			for some $r, \rho, s$. In this case, the trinomial
			\begin{equation}\label{trii} 
			r x^{q^m+1}+\rho x +s
			\end{equation}
			can have $0,1,2$ or $q+1$ roots. Moreover, the following holds
			$$|\Gamma| \in \{ q^{2n}-(-q)^{n+1}+1, q^{2n}-(-q)^n+1,   q^{2n}+1\}.$$
			\item $\Gamma$ is projectively equivalent to
			\[x_1^{\sigma+1}+ x_1 x_2^\sigma+r x_2^{\sigma+1}+\tau x_2 x_3^\sigma+sx_3^{\sigma+1}=0\]
			for some $r, \tau, s$.
			Moreover  $|\Gamma| \in \{q^{2n}\pm q^n+1, q^{2n}+1\}$.
		\end{enumerate}
		In the first case,
		\begin{itemize}
			\item if Equation $(\ref{trii})$ has either $q+1$  zeros or no zeros, then $\Gamma$ is equivalent to a set, which is already studied in the diagonal case (Cor. 7, Prop. 29, 30 in [4]),
			\item if Equation $(\ref{trii})$ has one zero, then 
			\[ |\Gamma| \in \{q^{2n}-q^{n+1}+1, q^{2n}+q^n+1\} \mbox{ if $n$ is odd}\] 
			\[ |\Gamma| \in \{q^{2n}+q^{n+1}+1, q^{2n}-q^n+1\} \mbox{ if $n$ is even,}\] 
			\item if Equation $(\ref{trii})$ has two zeros, then $|\Gamma|=q^{2n}+1$.
		\end{itemize}
	\end{theorem}

	\noindent For the results and the examples of this section we refer to the following papers by B.C. Kestenband
	\cite{Kes90,Kes00,Kes03,Kes05_1,Kes05_2,Kes05_3,Kes06,Kes07,Kes09,Kes10,Kes14}.



\begin{thebibliography}{99}

\bibitem{Bal} S. Ball, Finite Geometry and Combinatorial Applications. London Mathematical Society Student Texts, Cambridge University Press, (2015).

\bibitem{Cam} P.J.  Cameron, Projective and polar spaces. University of London. {\em Queen Mary and Westfield College}, (1992).


\bibitem{CosMarPav}  A. Cossidente, G. Marino, F. Pavese, Non-linear maximum rank distance codes. {\em Des. Codes Cryptogr.}, 79 (3), (2016), 597--609.


\bibitem{Del} Ph. Delsarte, Bilinear forms over a finite field, with applications to coding theory.
{\em J. Combin. Theory Ser. A} 25, (1978), 226--241.


\bibitem{Dem} P. Dembowski, Finite geometries. Berlin, Heidelberg, New York: Springer-Verlag, (1997), (reprint of the 1968 edition).


\bibitem{DonDur03} G. Donati, N. Durante,
Some subsets of the Hermitian curve. {\em European J. of Combinatorics} 24 (2), (2003), 211--218.


\bibitem{DonDur05} G. Donati, N. Durante,
Baer subplanes generated by collineations between pencils of lines. {\em Rendiconti del Circolo Matematico di Palermo}. Serie II. Tomo LIV, (2005), 93--100.


\bibitem{DonDur14} G. Donati, N. Durante, Scattered linear sets generated by collineations between pencils of lines. 
{\em J. Algebraic Combin.} 40 (4), (2014), 1121--1134.


\bibitem{DonDur18}  G. Donati, N. Durante, 
A generalization of the normal rational curve in $\PG(d,q^n)$ and its associated non-linear MRD codes. {\em Des. Codes Cryptogr.} 86 (6), (2018), 1175--1184. 

\bibitem{Dur} N. Durante, Geometry of sesquilinear forms. Notes of a course given at the Summer School
FINITE GEOMETRY \& FRIENDS, Brussels,  (2019).

\bibitem{DurSic} N. Durante, A. Siciliano, Non-linear maximum rank distance codes in the cyclic model for the field reduction of finite geometries. {\em Electron. J. Combin.} 24 (2), (2017), Paper 2.33.


\bibitem{Gru-Wei} K.W. Gruenberg,  A.J. Weir, Linear Geometry. Graduate Texts in Mathematics, 49 (1st ed.), Springer-Verlag New York, (1977).


\bibitem{HelKho} T. Helleseth, A. Kholosha, $x^{2^l+1}+x+a$ and related affine polynomials over $\GF(2^k)$.
{\em Cryptogr. Commun.} 2, (2010), 85--109.

\bibitem{HugPip} D.R. Hughes, F.C. Piper, Projective planes. Graduate Texts in Mathematics, Vol. 6. Springer-Verlag, New York-Berlin, (1973).


\bibitem{Kes90} B.C. Kestenband,  Correlations whose squares are perspectivities. {\em Geom. Dedicata} 33 (3), (1990), 289--315. 

\bibitem{Kes00} B.C. Kestenband,  The correlations with identity companion automorphism, of finite Desarguesian planes. {\em Linear Algebra Appl}. 304 (1-3), (2000),  1--31.

\bibitem{Kes03} B.C.   Kestenband, The correlations of finite Desarguesian planes. I. Generalities. {\em J. Geom}. 77 (1-2), (2003),  61--101. 

\bibitem{Kes05_1} B.C.  Kestenband, The correlations of finite Desarguesian planes of even nonsquare order. {\em JP J. Geom. Topol}. 5 (1), (2005),  1--62.

\bibitem{Kes05_2} B.C. Kestenband,  The correlations of finite Desarguesian planes. II. The classification (I). {\em J. Geom}. 82 (1-2), (2005), 91--134. 

\bibitem{Kes05_3} B.C.  Kestenband, The correlations of finite Desarguesian planes. III. The classification (II). {\em J. Geom}. 83 (1-2), (2005), 88--120.

\bibitem{Kes06} B.C. Kestenband,  The correlations of finite Desarguesian planes. IV. The classification (III). {\em J. Geom}. 86 (1-2), (2006),  98--139.

\bibitem{Kes07} B.C.  Kestenband,  The correlations of finite Desarguesian planes of square order defined by diagonal matrices. {\em Linear Algebra Appl.} 423 (2-3), (2007), 366--385. 

\bibitem{Kes09} B.C. Kestenband,  Embedding finite projective geometries into finite projective planes. {\em Int. Electron. J. Geom}. 2 (2), (2009),  27--33. 

\bibitem{Kes10} B.C. Kestenband,  The correlations of finite Desarguesian planes of odd square order defined by upper triangular matrices with four nonzero entries. {\em JP J. Geom. Topol}. 10 (2), (2010),  113--170. 

\bibitem{Kes14} B.C.  Kestenband, The correlations of finite Desarguesian planes of even square order defined by upper triangular matrices with four nonzero entries. {\em JP J. Geom. Topol}. 16 (2), (2014),  127--152. 


\bibitem{LawVan} {\rm M. Lavrauw, G. Van de Voorde}, {\rm On linear sets on a projective line.} {\it Des. Codes Cryptogr.} {56}, (2010), 89--104.


\bibitem{LidNie}
R. Lidl, H. Niederreiter, Finite Fields. (2nd ed.) Cambridge University Press, (1997).


\bibitem{Lun-Mar-Pol-Tro} {\rm G. Lunardon, G. Marino, O. Polverino, R. Trombetti}, {\rm Maximum scattered linear sets of pseudoregulus type and Segre Variety ${\cal S}_{n,n}$}. {\it J. Algebr. Comb.} {39} (2014), 807--831.

\bibitem{Ste} {\rm J. Steiner}, {\rm Systematische Entwichlung der Abh$\ddot{a}$ngigkeit Geometrische Gestalten von
	einander.}  { Reimer, Berlin (1832)}.

\end{thebibliography}
\end{document}